\documentclass[11pt]{amsart}
\usepackage{amsmath,amssymb,latexsym}
\usepackage[latin1]{inputenc}
\usepackage{amssymb, amsmath}

\theoremstyle{plain}
\newtheorem{thm}{Theorem}[section]

\newtheorem{fact}[thm]{Fact}
\newtheorem{lem}[thm]{Lemma}

\theoremstyle{definition}
\newtheorem{defn}[thm]{Definition}
\newtheorem{remark}[thm]{Remark}

\newtheorem*{expl}{Example}

\def\Ind{\setbox0=\hbox{$x$}\kern\wd0\hbox to 0pt{\hss$\mid$\hss}
\lower.9\ht0\hbox to 0pt{\hss$\smile$\hss}\kern\wd0}
\def\Notind{\setbox0=\hbox{$x$}\kern\wd0\hbox to 0pt{\mathchardef
\nn=12854\hss$\nn$\kern1.4\wd0\hss}\hbox to
0pt{\hss$\mid$\hss}\lower.9\ht0 \hbox to
0pt{\hss$\smile$\hss}\kern\wd0}

\def\dcl{\mathrm{dcl}}
\def\acl{\mathrm{acl}}
\def\cb{\mathrm{Cb}}
\def\tp{\mathrm{tp}}
\def\stp{\mathrm{stp}}
\def\Aut{\mathrm{Aut}}
\def\U{\mathrm{U}}

\def\Q{\mathcal Q}

\title[The strong canonical base property]{On definable Galois groups and the strong canonical base property}
\author{Daniel Palac\'\i n and Anand Pillay}
\date{January 13, 2016}
\address{Universit\"at M\"unster; Institut f\"ur Mathematische Logik und
Grundlagenforschung,
Einsteinstrasse 62, 48149 M\"unster, Germany}
\email{daniel.palacin@uni-muenster.de}
\address{University of Notre Dame; Departament of Mathematics, 281 Hurley Hall,
Notre Dame, IN 46556, USA}
\email{apillay@nd.edu}
\thanks{The first author was partially supported by the project Groups, Geometry \& Actions (SFB 878) and the project L\'ogica Matem\'atica (MTM2014-59178-P). The second author was supported by NSF grant DMS-1360702, and also thanks Rahim Moosa for some preliminary discussions on the subject of this paper.}

\begin{document}

\begin{abstract}
In  \cite{HPP}, Hrushovski and the authors proved, in a certain finite rank environment, that rigidity of definable Galois groups implies that $T$ has the canonical base property in a strong form; `` internality to"  being replaced by ``algebraicity in".  In the current paper we  give a reasonably robust definition of the ``strong canonical base property" in  a  rather more general finite rank context than \cite{HPP}, and prove its {\em equivalence}  with rigidity of the relevant definable Galois groups.  The new direction is an elaboration on the old result that $1$-based groups are rigid.

\end{abstract}

\maketitle

\section{Introduction}

In a stable theory, a set $X$ which is type-definable over a small set of parameters $A$,  is called $1$-based if for any tuple $b$ of elements of $X$ and set $B$ of parameters containing $A$, the canonical base of $\stp(b/B)$ is contained  in $\acl(b,A)$.  If $X = G$ happens to be a type-definable group then $1$-basedness of $G$ implies {\em rigidity} of $G$, which means that every connected type-definable subgroup of $G$ is type-definable over $\acl(A)$ \cite{HP}.  But rigidity of $G$ does not imply $1$-basedness: for example in ${\rm ACF}_{0}$ semiabelian varieties are rigid, but not $1$-based.

Recently, relative versions of one-basedness have come into play.
The main example is the canonical base property (CBP), which originates in \cite{pil02} and \cite{PZ}, was formally defined in \cite{MooPil}, and also studied by Chatzidakis \cite{zoe}, and in a more general framework by Palac\'in and Wagner \cite{pal-wag}, where (a weak version of) the CBP is indeed treated as a generalization of one-basedness.

The CBP, as formulated in \cite{MooPil} is a  property appropriate to {\em finite rank} stable theories $T$ and says (of $T$) that for any tuple $b$ and set of parameters $B$,  if $c$ is the canonical base of $\stp(b/B)$, then $\stp(c/b)$ is almost  internal to  the family of $\U$-rank $1$ types (or equivalently, to the  family of nonmodular $\U$-rank $1$  types).  The CBP  holds of the many-sorted theory ${\rm CCM}$ of compact complex spaces (see \cite{pil02}) as well as the finite Morley rank part of the  theory ${\rm DCF}_0$ of differentially closed fields of characteristic $0$ (see \cite{PZ}). Moreover, as pointed out in \cite{PZ},  this ${\rm DCF}_{0}$ case yields a quick account of function field Mordell-Lang in characteristic $0$.

In the joint paper with Hrushovski \cite{HPP} an example appears of a finite rank, in fact $\aleph_{1}$-categorical theory, where the CBP fails.  Now if $T$ is  an  $\aleph_{1}$-categorical theory, then the obstruction to $T$ being almost strongly minimal is given by infinite definable Galois groups (using the theory of analyzability and definable automorphism groups), of which precise definitions will be given below.  Similarly for finite rank nonmultidimensional theories.
In \cite{HPP} we worked, for convenience, with finite rank theories $T$ which are ``coordinatised"  by strongly minimal formulas over $\emptyset$, and proved that if all relevant definable Galois groups are rigid then $T$ has the CBP in the strong form that whenever
 $c = \cb(\stp(b/c))$, then $c$ is contained in the algebraic closure of $b$ and realizations of strongly minimal formulas over $\emptyset$.  (This was generalized in \cite{pal-wag2} using results in \cite{zoe}).  In the current paper we  give a robust definition of  $T$ having the {\em strong canonical base property} in a setting somewhat more general than that considered in \cite{HPP}, and prove its a equivalence with the rigidity of the relevant definable Galois groups (see Theorem \ref{ThmFinRank} in Section 3).
 In Section 2, we give a ``local" version of the equivalence between rigidity and the strong CBP  (see Theorem \ref{ThmRigid}).



We will assume familiarity with stability theory and geometric stability theory, and  the reader is referred to \cite{PilBook}. Nevertheless, we give here a brief discussion of some of the key notions of this paper.
We shall be working inside a monster model $\mathfrak M^{\rm eq}$ of a first-order stable theory $T$. Sometimes we talk about global types, namely complete types over this monster model. If $b$ is a tuple and $B$ a set of parameters then the canonical base of $\stp(b/B)$, which we denote $\cb(\stp(b/B))$,  is the smallest definably closed (in $\mathfrak M^{\rm eq}$) subset $C$ of $\acl(B)$ such that $b$ is independent from $\acl(B)$ over $C$ and $\tp(b/C)$ is stationary. We often identify $C$ with a tuple enumerating it, or with a tuple with which it is interdefinable.  Sometimes we write $\cb(b/B)$ for $\cb(\stp(b/B))$. When $T$ is totally transcendental $\cb(\stp(b/B))$ can be taken to be a finite tuple.

Let us  fix a family $\Q$ of partial types over  small sets of parameters. For $B$ a set of parameters, we say $a$ realizes $\Q$ over $B$ if there is a partial type $\Psi$ in $\Q$ over a set of parameters contained in $B$ and $a$ realizes $\Psi$.

Let $p$ be a stationary type over a set $A$. We say that $p$ is {\em (almost) internal to $\Q$} if there
is a superset $B$ of $A$, some realization $a$ of $p|B$ (the unique nonforking extension of $p$ over $B$), and a
tuple $\bar b$ of realizations of  $\Q$ over $B$ such that $a\in\dcl(B,\bar b)$ (or $a\in\acl(B,\bar b)$, respectively).

Internality gives rise to Galois groups (possibly trivial) typically when $p\in S(A)$ and $\Q$ is a family of partial types over $A$.  In this case we sometimes identify $\Q$ with the union of the sets of realizations of the partial types in $\Q$.  Let $\Q$ be such, and let $p\in S(A)$ be stationary and internal to $\Q$. Then according to Lemma 7.4.2 of \cite{PilBook}, there exists an
$A$-definable function $f(\bar x,\bar y)$, partial types types $\Psi_0,\ldots,\Psi_n$ in $\Q$ and
some tuple $\bar a$ of realizations of $p$ such that,  for any realization $a$
of $p$ there is a tuple $\bar c=(c_0,\ldots,c_n)$ with $c_i$ realizing $\Psi_i$
such that $a=f(\bar a,\bar c)$. The tuple $\bar a$ is usually called a {\em
fundamental system of solutions} of $p$ relative to $\Q$.   This is   required to
 obtain the following result, which is \cite[Theorem 7.4.8]{PilBook}, and due to Hrushovski at this level of generality, although in more special cases (such as $\aleph_{1}$-categorical theories) it was observed by Zilber.

\begin{fact}\label{FactGal}
If $\Q$ is a family of partial types over $A$ and $p$ is a stationary type over $A$ internal to $\Q$, then there is an
$A$-type-definable group $G$ and an $A$-definable action of $G$ on the set of
realizations of $p$ which is naturally isomorphic (as a group action) to the
group $\Aut(p/\Q,A)$ of permutations of the set of realizations of $p$ induced by the
automorphisms of $\mathfrak M$ fixing $A\cup \Q$ pointwise.
\end{fact}

We call the group $G$ from Fact \ref{FactGal} the {\em (type-)definable Galois group of $p$ relative to $\Q$}, and identify it with
$\Aut(p/\Q,A)$; sometimes it is also called the liason group or binding group.
Observe that whenever $p$ is internal to $\Q$ and $\bar a$ is a fundamental system
of solutions of $p$ relative to $\Q$, $\tp(\bar a/A)$ is also internal to $\Q$,  $G = \Aut(p/\Q,A)$ acts (definably over $A$)  on the
set of realizations of $\tp(\bar a/A)$, and moreover coincides with $\Aut(\tp(\bar a/A)/\Q,A)$. Also, an easy argument (see
Claim I of \cite[Theorem 7.4.8]{PilBook}) yields that

\begin{fact}\label{FactFree}
The type-definable group $\Aut(p/\Q,A)$ acts freely on the set of realizations of
$\tp(\bar a/A)$, that is for some/any realization $\bar b$ of $\tp(\bar a/A)$, and $\sigma\in  \Aut(p/\Q,A)$, $\sigma$ is determined by $\sigma(\bar b)$.
\end{fact}

This observation plays an essential role in the proof of Lemma \ref{PropGal}. On the other hand
observe that there is no reason why $\Aut(p/\Q,A)$ should act freely on
the set of realizations of $\tp(a/\Q,A)$ for $a$ a single realization of $p$.

If $p\in S(A)$ is stationary and internal to $\Q$ we say that $p$ is {\em fundamental} if some (any) realization of $p$ is already a fundamental system of solutions of $p$ relative to $\Q$.


\section{The local theory}
In this section we focus on internality and analyzabilty with respect to some given collection $\Q$ of partial types, and obtain key results relating rigidity of Galois groups to canonical bases.  In the final section this will be applied to suitable finite rank theories together with a canonical choice for $\Q$.

Recall that we assume the ambient theory to be stable, and let now $\Q$ denote a family of partial types with parameters over a set $A$.
In the proof below we will make use of the fact that (in a stable theory), type-definable groups have canonical parameters, namely for any type-definable group $H$ there is a (small) tuple of parameters $u$ such that an automorphism of the monster model fixes $H$ setwise iff it fixes the tuple $u$.  See Remark 1.6.20 of \cite{PilBook}. Likewise for type-definable homogeneous spaces (namely orbits under type-definable group actions).


\begin{lem}\label{PropGal}
Let $p = \tp(a/A)$ be a stationary complete type over $A$ which is internal to $\Q$, and fundamental.
Let $G$ be its type-definable over $A$ Galois group relative to $\Q$. Let  $H$ be a
type-definable connected subgroup of $G$, $u$ its canonical parameter,  and let $d$ be the   canonical
parameter of the orbit $H\cdot a$.  Then
\begin{enumerate}
\item $u\in \dcl(d,A)$
\item $\tp(a/A,d)$ implies $\tp(a/A,d,\Q)$
\item$\tp(a/A,d)$ is stationary and $d$ is interdefinable over $A$ with
\newline
$\cb(\tp(a/A,d))$.
\end{enumerate}
\end{lem}
\begin{proof} We will be using freely the basic material in  Chapter 1, Section 6, of \cite{PilBook} on  stable groups. For notational convenience we assume $A = \emptyset$.  Let $X$ be $H\cdot a$. Now by Fact 1.2,  $G$ is acting freely on the set of realizations of $p$, hence
$$
H=\{g\in G: g\cdot b = c \mbox{ for some  } b,c\in X\}.
$$ So $H$ is (type-)defined over $d$, whereby $u\in dcl(d)$, giving (1).

Suppose $\tp(b/d) = \tp(a/d)$. Thus $b\in X$ and so, there is some $h\in H$ such that $h\cdot a = b$. But the map taking $x\in X$ to $h\cdot x\in X$ is induced by an automorphism $\sigma$ of $\mathfrak M$ which fixes $X$ setwise and $\Q$ pointwise. Hence $\tp(b/d,\Q) = \tp(a/d,\Q)$, giving (2).

We have already seen that all  elements of $X$ have the same type over $d$.  But bearing in  mind part (1),  the principal homogeneous space $(H,X)$ is (type-)defined over $d$ and by connectedness of $H$ has a unique generic type over $d$ which is moreover stationary. So this unique generic type has to coincide with $\tp(a/d)$.  We have shown stationarity of $\tp(a/d)$.  Let $q$ be the unique global (i.e. over $\mathfrak M$) nonforking extension of $\tp(a/d)$, namely $q$ is precisely the unique global generic type of the principal homogeneous space  $(H,X)$. Let $\sigma$ be an automorphism of $\mathfrak M$ such that $\sigma(q) = q$. To show that $d = \cb(\stp(a/d))$ we must show that $\sigma(d) = d$.  Suppose first that  $\sigma(u) = u$, so $\sigma(X)$ is also an orbit under $H$, so is either equal to $X$ or disjoint from $X$. As $\sigma(q) = q$, $\sigma(X)$ is not disjoint from $X$. Hence $\sigma(X) = X$, namely $\sigma(d) = d$.  On the other hand, suppose $\sigma(u) \neq u$, namely $\sigma(H) = H' \neq H$. Let $X' = \sigma(X)$ and note that $\sigma(q)=q$ implies that $X'\cap X\neq \emptyset$. So $X'\cap X$ is an orbit under $H'\cap H$. The latter is a proper type-definable subgroup of the connected group $H$, where by $q$ being the generic type of $X$ implies $x\notin X'\cap X$. This contradicts $q= \sigma(q)$ (and $\sigma(q)\in X'$).
\end{proof}



\begin{defn}\label{DefRigid}
An $A$-type-definable group $G$ is said to be {\em rigid} if any type-definable connected subgroup is type-defined over $\acl(A)$.
\end{defn}

\begin{lem}\label{LemEquiv}
Let $p=\tp(a/A)$ be a stationary type which is internal to $\Q$, and let $G$ be the Galois group of $p$ relative to $\Q$.

Consider the following conditions:
\begin{enumerate}
 \item $G$ is rigid
 \item For any tuple $c$,  $\cb(\stp(c/A,a))\in\acl(A,c,\Q)$.
\end{enumerate}

Then (1) implies (2).  Moreover if $p$ is also fundamental then (2) implies (1).
\end{lem}
\begin{proof} The fact that (1) implies
(2) is precisely \cite[Lemma 2.3]{HPP}. Here we take the opportunity to   give a somewhat cleaner presentation of the proof.

Again for simplicity of notation assume $A = \emptyset$.
Assume (1). Given a tuple $c$, the stationary
type $\stp(a/c)$ is also internal to $\Q$. Let $H$ be the connected component of its Galois group (relative to $\Q$). So $H$ is (naturally) a type-definable subgroup of $G$, which is, by assumption (1), type-defined over $\acl(\emptyset)$.  Let $X$ be the orbit $H\cdot a$ and let $d$ be its canonical parameter.
So $X$ is fixed setwise under the action of $H$ and moreover $H$ acts transitively on $X$, induced by automorphisms of $\mathfrak M$ which fix $\acl(c)$ and $\Q$ (pointwise).
It follows that all elements of $X$ have the same type over $d,c,\Q$.  In particular
$$
\tp(a/d)\models \tp(a/c,d)
$$
and so $a$ is independent from $c$ over $d$, so also $c$ is independent from $a$ over $d$. Now, as $H$ is type-definable over $\acl(\emptyset)$ have that $d\in\acl(a)$ and so
$$
\cb(\stp(c/a))=\cb(\stp(c/d)).
$$
On the other hand, clearly, $X$, the orbit of $a$
under $H$, is nothing else than the set of realizations of
$\stp(a/c,\Q)$, thus $d$ belongs to $\acl(c,\Q)$ and so does
$\cb(\stp(c/a))$, as desired.

\vspace{2mm}
\noindent
We now assume that $p$ is fundamental, and prove that (2) implies (1) which is, strictly speaking, the new  result.  So assume (2), and let $H$ be a connected type-definable subgroup of
$G$, with canonical parameter $u$. We aim to show that $u\in \acl(\emptyset)$.
 Take a realization $a$ of $p$ independent from $u$ over $\emptyset$, and let $d$  be the canonical parameter of the orbit of $a$ under $H$. We will be using Lemma \ref{PropGal}.

Now by assumption $\cb(\stp(d/a))\in \acl(d,\Q)\cap \acl(a)$. But by (2) of \ref{PropGal}, $a$ is independent from $(\Q,d)$ over $d$, whereby $\cb(\stp(d/a))\in \acl(d)$.  This means that $d$
and $a$ (and so also $a$ and $d$) are independent over $\acl(d)\cap \acl(a)$.  But by (3) of \ref{PropGal}, $d = \cb(\tp(a/d))$ from which we conclude that $d\in \acl(a)$.  Hence by (1) of
\ref{PropGal}, $u\in \acl(a)$. But $u$ and $a$ are independent over $\emptyset$, hence $u\in \acl(\emptyset)$ as required.  \end{proof}

 \medskip
We still fix a family $\Q$ of partial types over a base set of parameters $A$.
We emphasize that by  a {\em (type-)definable Galois group relative to $\Q$} we mean  $\Aut(p/\Q,B)$ for some stationary complete type $p(x)$ over some set $B$ of parameters which contains $A$.   In general $B, C,\ldots$ will denote sets of parameters containing $A$.

In order to prove the  main result of this section we need to introduce the following from \cite{pal-wag}. First, recall that a stationary type $p$ over $B$ is {\em
analyzable in $\Q$} if for any realization $a$ of $p$ there are
$(a_i:i<\alpha)\in\dcl(B,a)$ such that $\tp(a_i/B,a_j:j<i)$ is internal to $\Q$ for all $i<\alpha$, and
$a\in\acl(B,a_i:i<\alpha)$.
We denote by $\ell_1^\Q(a/A)$ the maximal (possibly infinite) tuple $b$ in $\acl(B,a)$ such that $\tp(b/B)$ is internal to $\Q$. An inspection of the proof of \cite[Theorem 3.4]{pal-wag} yields:

\begin{fact}\label{FactDomEquiv} The tuple $\ell_1^\Q(a/B)$ dominates $a$ over $B$: whenever $\ell_1^\Q(a/B)$ is independent from $c$ over $B$, then so is $a$.
\end{fact}

Now we can state and prove our main result of this section, which generalizes \cite[Theorem 2.5]{HPP}.

\begin{thm}\label{ThmRigid}
The following are equivalent, where $\Q$ is assumed to be a family of partial types over $A$.
\begin{enumerate}
 \item All Galois groups relative to $\Q$ are rigid.
 \item For any $B$ containing $A$,  if $\tp(a/B)$ is analyzable in $\Q$, then for any tuple $c$, $\cb(\stp(c/B,a))$ is contained in $\acl(B,c,\Q)$.
\end{enumerate}
\end{thm}

Notice that (1) implies (2) was already (essentially)  proved in \cite[Theorem 2.5]{HPP}. For the sake of completeness we include a proof.

\begin{proof}
Suppose first that all Galois groups relative to $\Q$ are rigid, let $c$ be an arbitrary tuple and consider a
type $\tp(a/B)$ which is analyzable in $\Q$. Let $C$ be the set
$\acl(B,a)\cap\acl(c,B,\Q)$ and set $a_1$ to be $\ell_1^\Q(a/C)$. Note that $\acl(B,a)=\acl(C,a)$ and that $\stp(a_1/C)$ is internal to $\Q$ by definition. As the Galois group of $\stp(a_1/C)$ relative to $\Q$ is rigid by assumption, the canonical base $\cb(\stp(c/C,a_1))$ is algebraic over $B,c,\Q$ by Lemma \ref{LemEquiv}. Thus
$$
\cb(\stp(c/C,a_1))\subseteq\acl(B,a)\cap\acl(B,c,\Q)=\acl(C)
$$
and hence $c$ is independent from $a_1$ over $C$. Whence, $a_1$ and $a$ are domination-equivalent over $C$ by Fact \ref{FactDomEquiv} since $\tp(a/B)$ is analyzable in $\Q$, and thus $c$ is independent from $a$ over $C$. Therefore
$$
\cb(\stp(c/B,a))=\cb(\stp(c/C,a))\subseteq C\subseteq\acl(B,c,\Q)
,$$
as desired.

Assume now that condition (1) fails  holds,  and let $G$ be the Galois group relative to
$\Q$ of some stationary type $\tp(a/B)$, which is internal to $\Q$, such that $G$ is not rigid.  By the discussion in section 1 we may assume that $p$ is fundamental.  By  Lemma \ref{LemEquiv} (where $A$ there corresponds to $B$ here), there is a tuple $c$ such that $\cb(\stp(c/B,a))$ is not contained in $\acl(B,c,\Q)$, contradicting (2).
\end{proof}

\begin{remark} In the above Theorem we can replace condition (2) by:
\newline
(3) :  For any $a,c$ and $B\supseteq A$, if the type $\tp(\cb(\stp(c/B,a)/B))$ is analyzable in $\Q$, then
$\cb(\stp(c/B,a))$ is contained in $\acl(B,c,\Q)$.
\end{remark}
\begin{proof} It is clear that (3) implies (2). Moreover, (3) follows easily from (2) by noticing that
$$\cb(s\tp(c/A,a))=\cb\big(\tp(c\big/\cb(\stp(c/A,a)))\big).$$
\end{proof}

\section{The strong canonical base property and the global theory}

In \cite{HPP} we worked with (stable) theories which are ``coordinatised" by strongly minimal sets without parameters. We will work here with a broader class of finite rank theories:

\subsection*{Assumption} $T$ is a complete superstable theory such that every type has finite $\U$-rank, and such that every non locally modular type of $\U$-rank $1$ is nonorthogonal to $\emptyset$ (namely $T$ is ``nonmultidimensional" with respect to non locally modular types of $\U$-rank 1). \newline

Among such theories are: any $\aleph_{1}$-categorical theory, the finite rank part of ${\rm DCF}_{0}$, and ${\rm CCM}$.
Of course there exist superstable theories of finite rank which do not satisfy the {\bf Assumption}; take a definable family of algebraically closed fields indexed by a definable set with no structure. But we are unaware of any {\em natural} example of a finite rank superstable theory which does not satisfy the assumption. Let us note that in a superstable environment any non locally modular type of $\U$-rank $1$ has Morley rank $1$ (by  Buechler's theorem, see Corollary 3.3 in Chapter 2 of \cite{PilBook}).

We will now make a canonical choice of $\Q$.  We define $\Q$ to be the collection of all  complete types $\tp(a/\emptyset)$ such that $\stp(a/\emptyset)$ is  internal to the family $\mathcal S$ of non locally modular stationary types (over arbitrary parameters) of $\U$-rank $1$.  Bearing in mind the remark above that types in $\mathcal S$ have Morley rank $1$, we could as well take $\Q$ to be the family of formulas $\theta(x)$ over $\emptyset$ which are  internal to $\mathcal S$.  At the level of realizations the two possible $\Q$'s mentioned are identical.  As the proofs below will show there would be no essential difference by considering instead the types over $\emptyset$ which are almost internal to $\mathcal S$.

\begin{defn}
$T$ is said to have the {\em strong canonical base property}  (strong CBP)  if for all $a,b$, if $b = \cb(\stp(a/b))$ then $b\in \acl(a,\Q)$.
\end{defn}

We first prove that the strong CBP is invariant under naming constants.

\begin{lem} The strong CBP is invariant under naming parameters; that is $T$ has the strong CBP if and only if for some/any subset $B$ of $\mathfrak M$, the theory $T_{B}$ of $\mathfrak M$ with names for elements of $B$ has the strong CBP.
\end{lem}
\begin{proof}  Given a set $B$ of parameters, let $\mathcal S_B$ denote the family $\mathcal S$ in the theory $T_B$, i.e. the family of non locally modular stationary $\U$-rank $1$ types over arbitrary parameters containing $B$. Thus $\mathcal S_B$ is contained in $\mathcal S$. On the other hand, let $\Q_{B}$ correspond to $\Q$ in the theory  $T_{B}$, which amounts to it being the set of (strong) types over $B$ which are internal to ${\mathcal S}_B$. Notice that if a type $\tp(c)$ is in $\Q$ then $\tp(c/B)$ belongs to $\Q_{B}$, so it follows from the definitions that if $T$ has the strong CBP then $T_{B}$ does too.

Now for the other direction.  Let $B$ be a set of parameters and assume that $T_{B}$ has the strong CBP. It is easy to see that for any $B'$ with the same type as $B$ over $\emptyset$, $T_{B'}$ has the strong CBP. Now let $a,b$ be such that $b= \cb(\stp(a/b))$. We want to show that $b\in \acl(a,\Q)$. By the above remark, we may assume that $B$ is independent from $a,b$ over $\emptyset$. Now $b$ is still $\cb(\stp(a/b,B))$, so $b\in \acl(a,B,d)$ for some tuple $d$ of realizations of $Q_{B}$. Now consider $c = \cb(\stp(d,B/a,b))$.  We have that $c\in \dcl(d_0,B_0,d_{1},B_{1},\ldots,d_{k},B_{k})$ for some Morley sequence $(d_0B_0, d_{1}B_{1},\ldots)$  in $\stp(dB/ab)$ with $d_0B_0=dB$. As $ab$ is independent from $B$ it follows that $ab$ is independent from $B_0B_{1}\ldots B_{k}$, and so we get that $c$ is independent from $B_0B_1\ldots B_k$ as $c\in\acl(ab)$. Since each $\tp(d_i/B_i)$ is internal to $\mathcal S_{B_i}$ and $\mathcal S_{B_i}$ is contained in $\mathcal S$,  each $\tp(d_i/B_i)$ is internal to $\mathcal S$ and so is $\tp(c)$. Hence $\tp(c)$ belongs to $\Q$. Moreover, by definition of $c$, we get that $ab$ is independent from $dB$ over $c$ and hence, as $b\in \acl(B,a,d)$, $b\in \acl(a,c)$, completing the proof.
\end{proof}

\begin{lem} Let $p\in S(B)$ be a stationary type. Then $p$ is analyzable in $\mathcal S$ iff $p$ is analyzable in  $\Q$.
\end{lem}
\begin{proof} Suppose first that  $p(x)\in S(B)$ is internal to $\Q$. As every type in $\Q$ is internal to $\mathcal S$, an easy argument yields that $p$ is internal to $\mathcal S$ as well, see Remark 4.3 of Chapter 7 of \cite{PilBook}. 
This shows that any stationary type which is analyzable in $\Q$ is analyzable in $\mathcal S$.

 For the converse, suppose first that $p(x)\in S(B)$ is stationary and internal to $\mathcal S$. We will show that $p$ is almost internal to $\Q$, which is enough to show that analyzability in $\mathcal S$ implies analyzability in $\Q$ by Lemma 1.3(i) of Chapter 8 of \cite{PilBook}.

Let $M\supset B$ be a model, and $a$ a realization of $p' = p|M$ (unique nonforking extension of $p$ to $M$), such that $a\in \dcl(M,b_{1},\ldots,b_{n})$ where each $b_{i}$ realizes some $q_{i}\in {\mathcal S}$ whose domain is in $M$.   As each $q_{i}$ is nonorthogonal to $\emptyset$ we may assume (by enlarging $M$) that there is some $c$ such that $\tp(c/M)$ does not fork over $\emptyset$ and each $b_{i}$ forks with $c$ over $\emptyset$ (so  $b_{i}\in \acl(M,c)$). Let ${\bar b} = (b_{1},\ldots,b_{n})$, and let $c_{0}$ be $\cb(\stp({\bar b}M/c))$. 
 So $c_{0}\in \acl(c)\cap \dcl({\bar b},M,{\bar b}_{1},M_{1},\ldots,{\bar b},M_{k})$ where $({\bar b}M, {\bar b}M_{1},\ldots,{\bar b}_{k}M_{k})$ is a suitable $c$-independent set of realizations of $\stp({\bar b}M/c)$.
As $c_{0}\in \acl(c)$ and $M$ is independent from $c$ over $\emptyset$ it follows that $c_{0}$ is independent from ($M,M_{1},\ldots,M_{k})$ over $\emptyset$ and hence $\tp(c_{0}/\emptyset)$ is internal to $\mathcal S$, so is in $\Q$.  As each $b_{i}\in \acl(M,c_{0})$, it follows that $a\in \acl(M,c_{0})$. So $\tp(a/B)$ is almost internal to $\Q$.
\end{proof}



Remember that a Galois group relative to $\Q$ is precisely something of the form $\Aut(p/B,\Q)$ for some stationary $p = \tp(a/B)$ which is internal to $\Q$. Owing to the members of $\mathcal S$ having Morley rank $1$, it follows that any such Galois group is a definable (rather than type-definable) group and moreover has  finite Morley rank.  So rigidity of this group means that all connected definable subgroups are defined over the base set $B$.

The following is our main theorem, which uses the results in the previous section as well as results from \cite{zoe}.

\begin{thm}\label{ThmFinRank}
The following are equivalent:
\begin{enumerate}
\item All Galois groups relative to $\Q$ are rigid.
\item The strong CBP holds.
\end{enumerate}
\end{thm}
\begin{proof}
(1) implies (2). Assume (1). Let $\tp(a/b)$ be stationary and assume $b$ is its canonical base.  By Theorem 1.16 of \cite{zoe} (see also \cite{pi01}),  the type $\tp(b/\acl(a)\cap \acl(b))$ is analyzable in $\mathcal S$. By Theorem \ref{ThmRigid}, (1) implies (2),  we have that $b\in \acl(a, (\acl(a)\cap \acl(b)), \Q$). As $\acl(a)\cap \acl(b)\subseteq \acl(a)$, it follows that $b\in \acl(a,\Q)$ giving (2).

Conversely, suppose that some Galois group relative to $\Q$ is not rigid. Then by Theorem \ref{ThmRigid} again, for some $B$ and $\tp(a/B)$ analyzable in $\Q$, and tuple $c$,
$\cb(\stp(c/B,a)$ is not contained in $\acl(B,c,\Q)$. In particular writing $b$ for $\cb(\stp(c/B,a))$, $b\notin \acl(c,\Q)$, so the strong CBP fails.
\end{proof}

\begin{expl} Consider the $2$-sorted structure $M$ consisting of the field $\mathbb C$ of complex numbers, together with an $n$-dimensional vector space $V$ over $\mathbb C$, where the language on the sort $\mathbb C$ is the unitary ring language, the language on $V$ is the abelian group language, and there is a function in the language for scalar multiplication,  ${\mathbb C}\times V$ to $V$.  Let $T = {\rm Th}(M)$.  Then $V$ is internal to $\mathbb C$ and  ${\rm Aut}(V/{\mathbb C})$ is ${\rm GL}_n(V)$ which is not rigid. However in this case (as $V$ is internal to $\mathbb C$),  $\Q$ is everything, and so there are NO nontrivial Galois groups relative to $\Q$ whereby $T$ has the strong CBP.
\end{expl}

\begin{expl} Let $T$ be the finite rank part of ${\rm DCF}_{0}$, namely the ($\omega$-stable) theory of the many sorted structure whose sorts are the $\emptyset$-definable sets of finite Morley rank in a given model of ${\rm DCF}_{0}$ (of course equipped with all induced structure).  Then $T$ does not have the strong CBP.
\end{expl}
\noindent
{\em Explanation.}  As the field $\mathcal C$ of constants (a sort in $T$) is the unique non locally modular strongly minimal set in $T$, up to nonorthogonality, it follows that internality to $\Q$ is equivalent to internality to $\mathcal C$, in $T$.  Now differential Galois theory gives us some set of parameters $B$ and a Galois group $G$ over $B$ relative to $\mathcal C$  which is definably isomorphic to ${\rm GL}(n,{\mathcal C})$ for  $n\geq 2$. So $G$ is clearly nonrigid.  The above remarks, together with Theorem 3.3  says that $T$ does not have the strong CBP.

\vspace{5mm}

We conclude with some comments and questions.  Firstly, we could weaken the {\bf Assumption} of this section,  by dropping the finite rank hypothesis, and then define $T$ to have the strong CBP if whenever $\tp(a)$ has finite $\U$-rank then for any $b$ such that $b=\cb(\stp(a/b))$, $b\in \acl(a,\Q)$.  Theorem 3.4 goes through, but more care would have to be taken with Lemma 3.2.
Secondly, the results of this paper should go through with suitable formulations in the (super) simple case;  the problem with a generalization to simple theories is that Galois groups are only {\em almost
hyperdefinable}.  Thirdly, it seems natural ask whether there is a Galois-theoretic interpretation or account of the CBP.

\bibliographystyle{plain}

\end{document}